\documentclass[11pt]{amsart}

\usepackage{amsmath,amssymb,graphicx,xspace}

\usepackage{amsthm, ifthen, amsfonts, epsfig, graphics, latexsym, psfrag, verbatim, textcomp, tikz, slashbox}

\newtheorem{thm}{Theorem}[section]
\newtheorem{lem}[thm]{Lemma}
\newtheorem{cor}[thm]{Corollary}

\newtheorem{prop}[thm]{Proposition}
\newtheorem{main}[thm]{Main Theorem}
\theoremstyle{definition}
\newtheorem{exmp}[thm]{Example}
\newtheorem{defn}[thm]{Definition}

\newcommand{\bbb}[1]{\ensuremath{\mathbb{#1}}}

\newcommand{\euc}{\bbb{E}}

\newcommand{\sph}{\bbb{S}}

\newcommand{\cat}{\textsc{CAT}}

\begin{document}

\title{Group Actions on $\cat(0)$ Simplicial $3$-Complexes}

\author[R.~Levitt]{Rena M.H. Levitt}
\address{Mathematics Department\\
        Pomona College\\
         610 North College Avenue\\
Claremont, CA 91711}
\email{rena.levitt@pomona.edu}

\date{\today}

\begin{abstract} The notions of nonpositive curved spaces and biautomatic groups are generalizations of the geometric properties of hyperbolic spaces and computational properties of their fundamental groups. Given the mutual origins of these conditions, one might conjecture that groups acting on nonpositive curved spaces are biautomatic. While the conventional wisdom is that counterexamples should exist, some groups acting on nonpositively curved spaces have been shown to be biautomatic. This article adds to the list of positive examples by proving that groups acting on $\cat(0)$ simplicial $3$-complexes are biautomatic.\end{abstract}

\maketitle

\section{Introduction}

It is a well understood principle in geometric group theory that there is the close connection between the intrinsic
geometry topological space and the computational properties of its fundamental group. A prototypical example of
such a connection is the fact that a closed, compact $n$-dimensional Riemannian manifold with strictly negative sectional curvatures has a contractible universal cover with unique geodesics, and a fundamental group whose word problem can be solved in linear time.These properties are not independent of one another. The linear time solution to the word problem is a consequence of the coarse hyperbolic geometry of the universal cover. Both properties have been generalized. The geometric properties lead to the theory of nonpositively curved and $\cat(0)$ spaces  initiated by Cartan, Alexandrov, and Toponogov in the early part of the twentieth century. The computational properties inspired the theory of Gromov hyperbolic groups and then the more general theory of biautomatic groups  developed by Cannon, Epstein, Holt, Levy, Paterson, and Thurston in the late 1980s.

There is some evidence that these generalizations maintain their close connection. In particular, there are very restrictive classes
of nonpositively curved spaces that are known to have biautomatic fundamental groups. Two such classes are given by the following theorems.

\begin{thm}[Gersten-Short]\label{thm:gs1} Let $K$ be a $\cat(0)$ triangle complex. Then $\pi_1(K)$ is biautomatic. \end{thm}

\begin{thm}[Januszkiewicz - \`{S}wi\c{a}tkowski]\label{thm:js} Groups acting geometrically on a systolic complexes are biautomatic.\end{thm}

The term triangle complex refers to a piecewise Euclidean $2$-complex where each $2$-cell is isometric to an equilateral triangle with unit side length. Gersten and Short's results was originally stated in the language of small cancelation theory, and has been restated them using the terminology of this article\cite{GeSh91}. Systolic complexes are the main object of study in Janusszkiewicz and \'{S}wi\c{a}tkowski's theory of simplicial nonpositive curvature, or SNPC. The SNPC condition is similar to the $\cat(0)$ property, but replaces the metric condition with a combinatorial condition on links\cite{JaSw06}. While $\cat(0)$ spaces and SNPC spaces share many features, neither condition implies the other in general. Both results rely on the existence of a canonical system of combinatorial paths. It remains an open question whether groups acting on arbitrary nonpositively curved spaces are biautomatic, although the conventional wisdom is that counterexamples should exist. 

This article adds to the examples of biautomatic groups acting on $\cat(0)$ spaces. The following theorem is the main result.

\begin{main}\label{main:bi}Groups acting properly on $\cat(0)$ simplicial $3$-complexes are biautomatic.
\end{main}

\noindent To prove the main theorem a path system that generalizes the set  of paths defined by Gersten and Short for $\cat(0)$ triangle complexes is shown to exist in $\cat(0)$ $3$-complexes. It relies on Katherine Crowley's proof on the existence of minimal spanning disks \cite{Cr08} and uses techniques developed by Jon McCammond and myself to analyze triangle-square complexes \cite{LeMc09}. In the interest of full disclosure, I would like to note that using the existence of spanning disks it can be shown that regular $\cat(0)$ simplicial $3$-complex are systolic and thus biautomatic by Theorem \ref{thm:js} (see Theorem \ref{thm:3simpsnpc}). The techniques used in this paper rely on the $\cat(0)$ structure of the complex, and are independent of  Janusszkiewicz and \'{S}wi\c{a}tkowski's argument.

This article is structured as follows. Section \ref{sec:geom3simp} is a brief review of the geometry of  $\cat(0)$ simplicial complexes. Section \ref{sec:combpaths} is devoted to the structure combinatorial paths and minimal spanning disks in simplicial $3$-complexes. The theory of biautomatic groups is reviewed in Section \ref{sec:bi}. Section \ref{sec:combgeo} is an analysis of the local structure of combinatorial geodesics in  $\cat(0)$ simplicial $3$-complexes. The canonical set of paths used to prove Main Theorem \ref{main:bi}is defined in Section \ref{sec:gsgeos}, and the final details of the proof are given in Section \ref{sec:biautomaticity}. Finally, the feasibility of extending the proof  of the main theorem to higher dimensions is discussed in Section \ref{sec:highdim}.


\section{$\cat(0)$ Simplicial Complexes}\label{sec:geom3simp}


This section is a brief review of piecewise euclidean and spherical simplicial complexes, and non positive curvature. Recall that a metric space is \emph{geodesic} if every pair of points is connected by a length minimizing path. 

\begin{defn}[Piecewise Euclidean and Spherical Complexes] A \emph{Euclidean polytope} is the convex hull of a finite set of points in euclidean space $\euc^{n}$. Similarly, a \emph{Spherical polytope} is the convex hull of a finite set of points contained in an open hemisphere of $\sph^{n}$. A \emph{piecewise Euclidean complex}, or \emph{PE-complex} is a cell complex built out of Euclidean polytopes glued together along faces by isometries.  A \emph{piecewise Spherical complex (PS-complex)} is a cell complex built out of spherical polytopes. A theorem of Martin Bridson's implies that piecewise Euclidean and Spherical complexes with finitely many cell isometry types are geodesic metric spaces \cite{BrHa99}. The dimension of a complex $K$ is the maximum of the dimensions of its cells if the maximum exists. If not, $K$ is infinite dimensional. The $n$-skeleton of a complex $K$, denoted $K^{(n)}$ is the union of $m$-cells of $K$ for $m \leq n$. A \emph{subcomplex} $L$ of a complex $K$ is a subset of $K$ that is also a complex. \end{defn} 

\begin{defn}[Simplicial Complex] An \emph{$n$-dimensional simplex} or \emph{$n$-simplex}  $\sigma$ is the convex hull of $n+1$ points in general linear position in $\euc^n$. A face of a simplex is the convex hull of a subset of the points defining $\sigma$. A $0$-simplex is a vertex, a $1$-simplex an edge, and a $2$-simplex a face. It can be useful to identify a simplex with its vertex set. If $\sigma^{(0)} = \{v_0, v_1, ..., v_n\}$, then the set $\{v_0, v_1, ..., v_n\}s$ \emph{spans} $\sigma$. A \emph{simplicial complex} is a piecewise Euclidean complex with each $n$-polytope isometric to an $n$-simplex. A simplicial complex is \emph{regular} if each edge has unit length. All simplicial complexes in this article are taken to be regular unless otherwise specified. A \emph{spherical simplicial complex} is a piecewise spherical complex with each cell isometric to a simplex in an open hemisphere of $\sph^{n}$. 
\end{defn}

 \begin{defn}[Flag and Full] A simplicial complex $K$ is \emph{flag} if every set of vertices pairwise connected by edges span a simplex of $K$. A subcomplex $L$ of a complex $K$ of is \emph{full} in $K$ if  $\sigma^{(0)} \subseteq L$ implies $\sigma \subseteq L$ for all $\sigma \subseteq K$. It immediately follows that Full subcomplexes of flag complexes are flag. \end{defn}

Intuitively, a $\cat(0)$ space is a geodesic metric with geodesic triangles ``thinner'' than their Euclidean counter parts.  For the purposes of this paper $\cat(0)$ will be defined in terms of Gromov's link condition. 

\begin{defn}[Metric Link] Let $\sigma$ be a $k$-face of an  $n$-simplex $\tau$. The \emph{metric link} of $\sigma$ in $\tau$ is  the set of unit tangent vectors orthogonal to $\sigma$ and pointing into $\tau$. This defines a spherical $(n-k-1)$-simplex. Let $\sigma$ be a cell of a simplicial complex $K$. The \emph{metric link} of $\sigma$ in $K$, denoted $lk_{K}(\sigma)$, is the spherical simplicial complex whose cells are the links of $\sigma$ in each $\tau \supseteq \sigma$ in $K$.  
\end{defn}

\begin{lem}\label{lem:linksfull} Let $L$ be a full subcomplex of a simplicial complex $K$. Then $lk_L(v)$ is full in $lk_K(v)$ for each vertex $v \subseteq L$. \end{lem}

\begin{proof} Suppose that $\sigma^{(0)} = \{v_0,v_1,...v_n\}\subseteq lk_L(v)$ for $\sigma \subseteq lk_K(v)$. Each vertex $v_i$ corresponds to an edge $e_i$ in $L$. Let $w_i$ be the vertex opposite $v$ on $e_i$ for $i = 0...n$. The existence of $\sigma \subseteq lk_K(v)$ implies that $\{v,w_1,w_2,...,w_n\}$ span a simplex $\tau$ in $K$. By fullness $\tau$ is contained in $L$. Thus $\sigma \subseteq lk_L(v)$.
\end{proof}

\begin{defn}[Nonpositively Curved, $\cat(0)$] Let $X$ be a geodesic metric space. A \emph{locally geodesic loop} $\rho$ is an embedding of a metric circle into $X$ satisfying the following property; at each point $x$ on the image of $\rho$ the angle between the incoming and outgoing tangent vectors of $\rho$ in $lk_{X}(x)$ is at least $\pi$. A piecewise Euclidean complex with finitely many isometry types of cells is \emph{nonpositively  curved} if the link of each cell contains no locally geodesic loops of length less than $2\pi$. If in addition $X$ is connected and simply connected, then $X$ is $\cat(0)$. \end{defn}

The intrinsic metric on a $\cat(0)$ space is convex, and geodesics are unique. The $\cat(0)$ and nonpositivel curvature conditions describe the behavior of metric links. Simplicial nonpositive curvature replaces the metric condition with a combinatorial condition on combinatorial links,

\begin{defn}[Star, Closed Star] The \emph {star of a simplex} $\sigma \subseteq K$, denoted $st(\sigma)$, is the union of the interiors of simplices $\tau$ containing $\sigma$. The \emph{closed star of $\sigma$} is $\bar{st}(\sigma) = \bigcup \{ \tau \subseteq K | \sigma \subseteq \tau \}$. \end{defn}

\begin{defn}[Combinatorial Link] The \emph{combinatorial link} of a simplex $\sigma$ in a simplicial complex $K$, denoted $clk(\sigma)$ is the union of all simplicies $\tau$ of $K$ such that $\sigma \ast \tau$ is a simplex of $K$. Thus $clk(\sigma) = \bar{st}(\sigma) \setminus st(\sigma)$.
\end{defn}

\begin{lem}\label{lem:clkfull} Let $\sigma$ be a simplex of $K$. Then $clk(\sigma)$ is a full subcomplex of a flag complex $K$.\end{lem}

\begin{proof} Suppose $\tau^{(0)} \subseteq clk(\sigma)$ for a simplex $\tau \subseteq K$. Then the set of vertices $\tau^{(0)} \cup \sigma^{(0)}$ are pairwise connected by an edge. As $K$ is flag, this implies $\tau \ast \sigma$ is a simplex of $K$. Thus $\tau \subseteq clk(\sigma)$. \end{proof}

\begin{defn}[Simplicial Nonpositive Curvature] A simplicial complex $K$ is \emph{$k$-large} if $K$ is flag and contains no empty $n$-gons for $n<k$. A simplicial complex $K$ satisfies the \emph{simplicial nonpositive curvature} condition, or \emph{$\textsc{SNPC}$},  if $clk_{K}(\sigma)$ is $6$-large  for each simplex $\sigma \subseteq K$. Simplicial complexes satisfying the SNPC condtion are \emph{systolic}. \end{defn}


\section{Combinatorial Paths and Spanning Disks}\label{sec:combpaths}

This section is a discussion combinatorial paths and disks in $\cat(0)$ simplicial complexes, including Katherine Crowley's result for spanning disks in $3$-complexes. The end of the section includes some useful consequences of her result.

\begin{defn}[Combinatorial Path, and Combinatorial Distance] A \emph{combinatorial path} is an alternating sequence of vertices and edges  $\alpha = [v_0,e_1,v_1, \\ e_2,..., v_{n-1},e_n, v_{n}]$ such that $e_i$ is spanned $\{v_{i-1},v_{i}\}$ for $1 \leq i <k$. The path is a \emph{loop} if $v_0 = v_n$. A combinatorial path is \emph{tight} if it does not cross the same edge twice. The length of a combinatorial path $\ell(\alpha)$ is the number of edges it crosses. The \emph{combinatorial distance} between two vertices $v$ and $w$,  denoted $d_c(v,w)$, is the minimum of the set of lengths of combinatorial path from $v$ to $w$. This defines a metric on $K^{(0)}$. A combinatorial path $\gamma$  from $v$ to $w$ is a \emph{combinatorial geodesic} if $\ell(\gamma) = d_c(v,w)$. Combinatorial geodesics are not unique.\end{defn}

\begin{defn}[Disk Diagram] A \emph{disk diagram} is a contractible $2$-complex that embeds in $\euc^2$. This embedding is often implicit. A disk diagram is \emph{nonsingular} if it is homeomorphic to the closed unit disk. Otherwise the disk contains a \emph{cut point} whose removal disconnects the diagram. Disks with cut points are called \emph{singular}. Singular disks may broken up into nonsingular subdisks. The boundary of a disk diagram is a combinatorial loop read clockwise around the outside of the disk. This may be ambiguous for a singular disk, where the boundary is determined by giving an explicit embedding in the plane. Suppose $D$ is a simplicial disk, i.e. $D$ has triangular faces. The \emph{combinatorial area} of $D$ is the number of faces contained in $D$.\end{defn}

\begin{defn}[Spanning Disk] Let $D$ be a simplicial disk diagram, $K$ be a simplicial complex, and $f: D \longrightarrow K$ a cellular map. Then $\alpha= f(\partial D)$  is a combinatorial loop in  $K$. In this case, $D$ \emph{spans} $\alpha$ in $K$ and $\alpha$ bounds $D$. \end{defn}

\begin{defn}[Vertex Degree] Let $D$ be a simplicial disk diagram. The \emph{degree} of a vertex $v$, denoted $deg_D(v)$ or simply $deg(v)$, is the number of edges sharing $v$ as a vertex. For vertices on the interior of $D$, this is equivalent to the number of faces with $v$ as a vertex. For boundary vertices, this is one more than the number of faces with $v$ as a vertex.  By Gromov's link condition a disk is $\cat(0)$ if and only if each interior vertex is contained in at least six triangles. \end{defn}

Many of the results in this article involve analyzing the structure of spanning disks. The Combinatorial Gau\ss-Bonnet Theorem will be used in many of these arguments. This is a classical result, a proof can be found in \cite{Cr08}.

\begin{thm}[Combinatorial Gau\ss-Bonnet] Let $D$ be a triangulated disk. Then
\begin{equation*}
    \sum_{v \in \partial D} (4-deg(v)) + \sum_{v \in intD} (6-deg(v)) = 6.
\end{equation*}
\end{thm}

By contractibility, each combinatorial loop bounds a spanning disk. The following theorem of Katherine Crowley's  determines the structure of these disks in $\cat(0)$ simplicial $3$-complexes \cite{Cr08}.

\begin{thm}[Crowley, Spanning Disks] \label{thm:spanningdisks} Let $K$ be a $\cat(0)$ simplicial $3$-complex, $\alpha$ be a combinatorial loop in $K^{(1)}$. Then there exists a $\cat(0)$ disk $D$  contained in $K$ of minimal combinatorial area such that $\partial D = \alpha$.
\end{thm}

\begin{cor}\label{cor:bdcount} Let $D$ be a minimal minimal spanning disk. Then  \begin{equation*} \sum_{v \in \partial D} (4-deg(v)) \geq 6. \end{equation*} \end{cor}

\begin{proof}Theorem \ref{thm:spanningdisks} implies that the degree of each interior vertex is at least six. Thus $\sum_{v \in intD} (6-deg(v)) <0$. \end{proof}

\begin{defn}[Empty $n$-gon]  A \emph{combinatorial $n$-gon} or \emph{$n$-gon} is a tight combinatorial loop of length $n$.  A $n$-gon $\alpha$ is \emph{empty} if the minimal disk spanning $\alpha$ has an interior vertex. For example, an empty triangle is a loop of length three not spanned by a face, an empty square is a loop of length four not spanned by two faces sharing an edge, and an empty pentagon is a loop of length five not is not spanned by faces. \end{defn} 

\begin{prop}\label{prop:noemptytrisqpent} $\cat(0)$ simplicial $3$-complexes contain no empty triangles, squares or pentagons. \end{prop}

\begin{proof} Let $\alpha = \{v_0, e_0, v_1, ..., v_{n} = v_0\}$ be an empty $n$-gon of minimal length,  $D$ be the minimal disk spanning $\alpha$. Since $\alpha$ is tight, $deg(v_i) >1$ for each $i$. If $deg(v_i) = 2$, then $v_i$ lies on a single triangle in $D$. Removing the triangle incident $v_i$ gives an $(n-1)$-gon $\alpha'$ spanned by $D' $ containing the same interior vertex as $D$, contradicting the choice of $\alpha$. Thus  $deg(v_i) \geq 3$ for all $i = 0, ..., n-1$. This implies $$ n  = \sum_{i=0}^{n-1} 1 \geq \sum_{i=0}^{n-1} (4-deg(v_i)) \geq 6.$$ \end{proof} 

\begin{lem}\label{lem:fullnotempty} Full subcomplexes inherit the no empty triangles, squares and pentagons conditions.
\end{lem}

\begin{proof} Let $L$ be a full subcomplex of a simplicial complex $K$ with no empty triangles, squares or pentagons. The no triangle condition follows immediately from the definition of full. Let $c = [v_1,v_2,v_3,v_4,v_1]$ be a closed combinatorial path of length four in $L$. Then $c$ is filled by a disk consisting of two triangular faces in $K$. Thus either $v_1$ and $v_3$, or $v_2$ and $v_4$ span an edge $e$ in $K$. The fullness condition implies $e$, and hence both faces, are contained in $L$. The same argument shows $L$ has no empty pentagons.
\end{proof}

\begin{thm}\label{thm:cat(0)flag} If $K$ is a $\cat(0)$ simplicial $3$-complex then $K$ is flag.
\end{thm}

\begin{proof}
Suppose that the vertices $v_1,v_2,...,v_n$ are pairwise connected by an edge in $K$. Proposition \ref{prop:noemptytrisqpent} implies that each triplet of vertices from the set $V = \{v_1,v_2,...v_n\}$ span a face. Let $W = \{w_1,w_2,w_3,w_4\}$ be a $4$-tuple of points of $V$. Then the full subcomplex of $K$ with vertex set $W$ contains the $2$-skeleton of a tetrahedron. As $\cat(0)$ spaces are contractible and contractions strictly reduce distance\cite{Bo92}, $W$ spans a tetrahedron of $K$. Let $U = \{u_1,u_2,u_3,u_4,u_5\} \subseteq V$ be distinct. Then the $1$-skeleton of the full subcomplex of $K$ with vertex set $U$ is a $K_5$ graph. Any $4$-tuple of points in $U$ span a tetrahedron. Let $e$ be an interior edge of the subcomplex. Then $e$ is contained in exactly three tetrahedra. Then $lk_K(e)$ contains a loop $c$ of three edges corresponding the three tetrahedra. The length of each edge is the dihedral angle at $e$ in each tetrahedron, $\arccos(\frac{1}{3})$. Thus $\ell(c) = 3\arccos(\frac{1}{3}) < \frac{3\pi}{2}$, a contradiction to the $\cat(0)$ condition. This implies that there can be at most four distinct vertices pairwise connected by edges and they must span a simplex of $K$. Thus $K$ is flag. 
\end{proof}

Theorem \ref{thm:cat(0)flag} together with Proposition \ref{prop:noemptytrisqpent}, Lemma \ref{lem:clkfull} and Lemma \ref{lem:fullnotempty} imply the following result.

\begin{thm}\label{thm:3simpsnpc} $\cat(0)$ simplicial $3$-complexes are systolic. \end{thm}


\section{Biautomatic Groups}\label{sec:bi}


This section is a brief overview of biautomatic groups. Essentially, a group is biautomatic if there exist a set of regular paths in the Cayley graph of the group obeying a distance condition. This is equivalent to the existence of a set of finite state automata that completely determine computation the group. See \emph{Word Processing in Groups} by David Epstein, et. al. \cite{ECHLPT92} for a thorough treatment of the subject. For the purposes of this paper, biautomicity will be defined in terms of path systems in graphs. The definition given in this section was developed by Jacek \'{S}wi\c{a}tkowski, and is shown to be equivalent to the traditional formulation in \cite{Sw06}. Note that the notions of combinatorial path and combinatorial distance may be extended to any connected graph by viewing the graph as a $1$-complex.

\begin{defn}[Parameterizing Paths] Let $\alpha=[v_0, v_1, ..., v_n]$ be a combinatorial path in a graph $\Gamma$. Then $\alpha$ may be thought of as a simplicial map from a subdivided interval into $\Gamma$. Define $\alpha(t)$ to be $v_0$ when $t \leq 0$, the point $t$ units from the $v_0$ along $\alpha$ when $0 \leq t \leq \ell(\alpha)$,  and $v_n$ when $t \geq \ell(\alpha)$. \end{defn}

\begin{defn}[Distance between Paths] . The \emph{path distance} between $\alpha$ and $\beta$ is maximum of the distance between $\alpha(t)$ and $\beta(t)$ for all $t$.\end{defn}

\begin{defn}[Path System] Let $P_{\Gamma}$ be the set of all paths in a graph $\Gamma$. A \emph{path system} is a subset $\mathcal{P}$ of $P_{\Gamma}$. The \emph{geodesic path system}  of $\Gamma$ is the set of geodesics in  $P_{\Gamma}$. \end{defn}

\begin{defn}[$(k,l)$-Fellow Travel] A path system $\mathcal{P}$ \emph{$(k,l)$-fellow travels} if given pair of paths $\alpha, \beta \in \mathcal{P}$ whose  endpoints are at most $l$ apart, the path distance between $\alpha$ and $\beta$ is at most $k$. If such a $k$ and $l$ exist, then $\mathcal{P}$ satisfies the fellow traveller property.\end{defn}

\begin{defn}[FSA and Regularity] A \emph{finite state automata}, or FSA, is a finite, labeled, directed graph $M$ with decorations. The vertex set of $M$ has a distinguished start vertex, and a set of distinguished accept vertices. A path is accepted by $M$ if it travels from the start vertex to an accept vertex. The set of paths accepted by $M$ is the \emph{path system induced by $M$} denoted $\mathcal{P}(M)$. A path system is \emph{regular} if it is accepted by some FSA. \end{defn}

\begin{thm}[Combing Regular Languages \cite{ECHLPT92}]\label{thm:combreg} Let $\mathcal{P}$ and $\mathcal{P'}$ be regular path systems. Then $\mathcal{P} \cup \mathcal{P'}$, $\mathcal{P} \cap \mathcal{P'}$, and the compliment of $\mathcal{P}$ are also regular path systems. \end{thm}

The following theorem is often used to show that path systems of geodesics are regular. It has been restated in the terminology of this paper.

\begin{thm}[Neumann-Shapiro]\label{thm:fft} Let $\Gamma$ be a graph such that each non geodesic path of $\Gamma$ is $(k,0)$-fellow traveled by a strictly shorter path. Then the geodesic path system of $\Gamma$ is regular. \end{thm}

\begin{defn}[Geometric Action] A group $G$ acts on a space $X$ \emph{geometrically} if the action is properly discontinuous, cocompact, and by isometries. \end{defn}

\begin{defn}[Biautomatic Group] Suppose a group $G$ acts geometrically on a graph $\Gamma$. Then $G$ is \emph{biautomatic} if there exist a path system $\mathcal{P}$ satisfying the following conditions.

\begin{enumerate}

\item $\mathcal{P}$ is $G$-invariant and regular.
\item There exist a vertex $v_0 \in \Gamma$ such that each path starts and ends at a vertex of the orbit $G(v_0)$ and $\mathcal{P}$ is transitive on $G(v_0)$.
\item $\mathcal{P}$ satisfies the fellow traveller property.

\end{enumerate}
\end{defn}

To show a group acting on a complex $K$ is biautomatic, one must find a path system in $\Gamma = K^{(1)}$ satisfying each of the above properties For $\cat(0)$ simplicial $3$-complexes, the path system used will be a canonical subset of the combinatorial geodesics in the $1$-skeleton of $K$. The definition of the system relies on determining the structure combinatorial geodesics between pairs of vertices.


\section{The Structure Combinatorial Geodesics}\label{sec:combgeo}


This section is devoted to the local structure of combinatorial geodesics. Spanning disks provide a concise method to move between a combinatorial path and a combinatorial geodesic with common endpoints by pushing across small disks. The machinery discussed in this section was developed by Jon McCammond and myself to determine the structure of combinatorial geodesics in triangle-square complexes \cite{LeMc09}. The terminology matches the original formulation, while the results and proofs have been reworked to apply to $\cat(0)$ simplicial $3$-complexes.

\begin{defn}[Doubly-Based Disk] A doubly based based disk diagram $D$ is a nonsingular disk diagram with two distinguished vertices, a start vertex  $u$ and  an end vertex $v$. This breaks the boundary of $D$ into two paths, one travelling clockwise from $u$ to $v$ and one traveling counterclockwise from $u$ to $v$. The clockwise path around $D$ is referred to as the \emph{old path}, and the counterclockwise path is the \emph{new path}. If both the old and new paths are both combinatorial geodesics, then $D$ is a \emph{geodesic} disk. \end{defn}

\begin{defn}[Moves] A \emph{move} is a doubly based diagram $D$ together with a map simplicial map $D \rightarrow K$. If the old path of $D$ is mapped to a directed subpath of $\gamma$ in $K$, then \emph{applying the move} $D$ to $\gamma$ is replacing the image of the old path in $\gamma$ with the image of the new path to obtain the path $\gamma'$ (see Figure\ref{fig:moves}). If $\ell(\gamma') = \ell(\gamma)$ then the move is length preserving, and if $\ell(\gamma') <\ell(\gamma)$ the move is length reducing.\end{defn}

\begin{center}

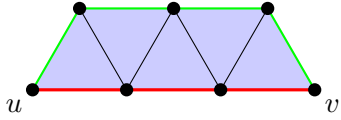
\begin{figure}

\begin{tikzpicture}[scale=1.25]

\coordinate (A0) at (0,0);
\coordinate (A1) at (1,0);
\coordinate (A2) at (2,0);
\coordinate (A3) at (3,0);

\coordinate (B0) at (0.5, 0.5*3^{0.5});
\coordinate (B1) at (1.5, 0.5*3^{0.5});
\coordinate (B2) at (2.5, 0.5*3^{0.5});

\filldraw[blue!20] (A0)--(A3)--(B2)--(B0)--cycle;

\draw[green, thick] (A0)--(B0)--(B2)--(A3);
\draw[red, very thick] (A0)--(A3);

\draw[thin] (B0)--(A1)--(B1)--(A2)--(B2);

\fill[black] (A0) circle (2pt) node[anchor=north east] {$u$};
\fill[black] (A1) circle (2pt);
\fill[black] (A2) circle (2pt);
\fill[black] (A3) circle (2pt) node[anchor=north west] {$v$};

\fill[black] (B0) circle (2pt);
\fill[black] (B1) circle (2pt);
\fill[black] (B2) circle (2pt);

\end{tikzpicture}

\caption{\label{fig:moves} An example of a move. The old upper green path is sent to the new lower red path. This move is length reducing.}

\end{figure}

\end{center}

In triangulated disks, there are three basic moves; trivial moves, triangle moves, and triangle-triangle moves (see Figure \ref{fig:basicmoves}). In a trivial move $u=v$, the disk is a single edge, the old path crosses the edge twice, and the new path is the constant path $u$. In a \emph{triangle move} the disk is a single triangle, old path consists of two sides of the triangle, and the new path is the third side. Finally, in a \emph{triangle-triangle move} the disk consists of two triangles sharing an edge, $u$ and $v$ are the opposing degree two vertices, the old path is clockwise path from $u$ to $v$, and the new path is the counterclockwise  path from $u$ to $v$. Trivial and triangle moves are length reducing, triangle-triangle moves are length preserving.  

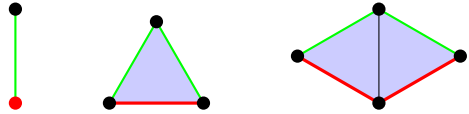
\begin{figure}

\begin{tikzpicture}[scale=1.25]

\draw[green,thick] (0,0) -- (0,1);

\fill[red] (0,0) circle (2pt);

\fill[black] (0,1) circle (2pt);

\filldraw[blue!20] (1,0)--(2,0)--(1.5, 0.5*3^{0.5})--cycle;

\draw[green,thick] (1,0)--(1.5, 0.5*3^{0.5})--(2,0);

\draw[red, very thick] (1,0)--(2,0);

\fill[black] (1,0) circle (2pt);
\fill[black] (2,0) circle (2pt);
\fill[black] (1.5,0.5*3^{0.5}) circle (2pt);

\filldraw[blue!20] (3, 0.5)--(3 + 0.5*3^{0.5}, 0) -- (3 + 3^{0.5}, 0.5) -- (3 + 0.5*3^{0.5}, 1) -- cycle;

\draw[green, thick] (3, 0.5)--(3 + 0.5*3^{0.5}, 1) -- (3 + 3^{0.5}, 0.5);
\draw[red, very thick] (3, 0.5)--(3 + 0.5*3^{0.5}, 0) -- (3 + 3^{0.5}, 0.5);

\draw (3 + 0.5*3^{0.5}, 0)--(3 + 0.5*3^{0.5}, 1);

\fill[black] (3, 0.5) circle (2pt);
\fill[black] (3 + 0.5*3^{0.5}, 0) circle (2pt);
\fill[black] (3 + 3^{0.5}, 0.5) circle (2pt);
\fill[black] (3 + 0.5*3^{0.5}, 1) circle (2pt);

\end{tikzpicture}

\caption{\label{fig:basicmoves}The three basic moves for triangulated simplicial disks: the trivial move, triangle move, and triangle-triangle move.}

\end{figure}

In doubly based spanning disk $D$,  points along the boundary of $D$ were a basic move could be applied are identified by the degree of the boundary vertex. Given a non distinguished  vertex $w \in \partial D^{(0)}$, if $deg(w) = 1$ a trivial move can be applied, if $deg(w) = 2$ a triangle move may be applied, and if $deg(w) = 3$ a triangle-triangle move may be applied.  This motivates the following definition.

\begin{defn}[Positive, Zero, and Negative Vertices] A vertex $v$ of a combinatorial path $\alpha$ in the boundary of a disk $D$ is a \emph{positive vertex} on $D$ if $4-deg_D(v)>0$, a \emph{zero vertex} if $4-deg_D(v) = 0$, and a \emph{negative vertex} if $4 - deg_D(v) <0$.  Let $\alpha_{+}(D)$, $\alpha_{0}(D)$, and $\alpha_{-}(D)$ denote the set of positive, zero, and negative vertices of $\alpha$ on $D$ respectively. Basic moves can be applied at positive vertices. To determine which move can be applied, the positive vertices can be further classified as $\alpha_1(D) = \{ v\in \alpha | 4-deg_D(v) = 1\}$, $\alpha_2(D) = \{ v\in \alpha | 4-deg_D(v) = 2\}$, and $\alpha_3(D) = \{ v\in \alpha | 4-deg_D(v) = 3\}$. Triangle-triangle moves can be applied at the vertices of $\alpha_1(D)$, triangle moves at vertices in $\alpha_2(D)$, and trivial moves at vertices in $\alpha_3(D)$. \end{defn}

Lemma \ref{lem:+vecount} is a consequence of Combinatorial Gau\ss-Bonnet. It was first observed by Gersten and Short in \cite{GeSh91}.

\begin{lem}\label{lem:+vecount}Let $D$ be a nonsingular $\cat(0)$ disk spanning a closed combinatorial path $\alpha$. Then $2|\alpha_2(D)| + |\alpha_1(D)| \geq |\alpha_-(D)| + 6$. If no triangle moves can be applied to $\alpha$, then $\alpha$ contains at least six more positive vertices than negative vertices. \end{lem}

\begin{lem} Let $\gamma = v_0,v_1,...,v_n$ be a combinatorial geodesic along the boundary of a disk $D$. Then $\sum_{i=1}^{n-1}(4-deg_D(v)) \leq 1$. \end{lem}

\begin{proof}  If $\gamma$ is a combinatorial geodesic, no length reducing moves are contained in the boundary of $D$ along $\gamma$. Thus $\deg_D(v_i) \geq 3$ for $i=1,...,n-1$. Reading along $\gamma$, if two vertices of $\gamma_{+}(D)$ are not separated by a vertex of $\gamma_{-}(D)$, then the length of $\gamma$ can be decreased by a sequence of triangle-triangle moves, followed by a single triangle move (see Figure \ref{fig:pathshortening}). Match each vertex of $\gamma_{+}(D)$, except possibly the last, with the subsequent vertex of $\gamma_{-}(D)$. Each pair contributes at most $0$ to the sum along $\gamma$. Suppose the final vertex of $\gamma_{+}(D)$ is $v_m$. Then $\sum_{i=1}^{n-1} (4-deg_D(v_i)) \leq (4-deg_D(v_m)) \leq 1$. 

\end{proof}


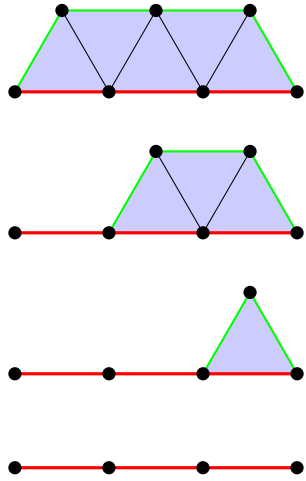
\begin{figure}
\begin{center}

\begin{tikzpicture}[scale = 1.25]

\coordinate (A1) at (0,3);
\coordinate (B1) at (3,3);
\coordinate (C1) at (2.5, 3+0.5*3^{0.5});
\coordinate (D1) at (0.5, 3+0.5*3^{0.5});
\coordinate (E1) at (1,3);
\coordinate (F1) at (1.5, 3+0.5*3^{0.5});
\coordinate (G1) at (2,3);

\filldraw [blue!20] (A1)--(B1)--(C1)--(D1)--cycle;

\draw[green, thick] (A1)--(D1)--(C1)--(B1);
\draw[red, very thick] (A1)--(B1);

\draw (D1)--(E1)--(F1)--(G1)--(C1);

\fill [black] (A1) circle (2pt);
\fill [black] (B1) circle (2pt);
\fill [black] (C1) circle (2pt);
\fill [black] (D1) circle (2pt);
\fill [black] (E1) circle (2pt);
\fill [black] (F1) circle (2pt);
\fill [black] (G1) circle (2pt);

\coordinate (A2) at (0,1.5);
\coordinate (B2) at (3,1.5);
\coordinate (C2) at (2.5, 1.5+0.5*3^{0.5});
\coordinate (D2) at (0.5, 1.5+0.5*3^{0.5});
\coordinate (E2) at (1,1.5);
\coordinate (F2) at (1.5, 1.5+0.5*3^{0.5});
\coordinate (G2) at (2,1.5);

\filldraw [blue!20] (E2)--(F2)--(C2)--(B2)--cycle;

\draw[green, thick] (E2)--(F2)--(C2)--(B2);
\draw[red, very thick] (A2)--(B2);

\draw (F2)--(G2)--(C2);

\fill [black] (A2) circle (2pt);
\fill [black] (B2) circle (2pt);
\fill [black] (C2) circle (2pt);
\fill [black] (E2) circle (2pt);
\fill [black] (F2) circle (2pt);
\fill [black] (G2) circle (2pt);

\coordinate (A3) at (0,0);
\coordinate (B3) at (3,0);
\coordinate (C3) at (2.5, 0.5*3^{0.5});
\coordinate (D3) at (0.5, 0.5*3^{0.5});
\coordinate (E3) at (1,0);
\coordinate (F3) at (1.5, 0.5*3^{0.5});
\coordinate (G3) at (2,0);

\filldraw [blue!20] (G3)--(C3)--(B3)--cycle;

\draw[green, thick]  (G3)--(C3)--(B3);

\draw[red, very thick]  (A3)--(B3);

\fill [black] (A3) circle (2pt);
\fill [black] (B3) circle (2pt);
\fill [black] (C3) circle (2pt);
\fill [black] (E3) circle (2pt);
\fill [black] (G3) circle (2pt);

\draw[red, very thick] (0,-1)--(3,-1);

\fill [black] (0,-1) circle (2pt);
\fill [black] (1,-1) circle (2pt);
\fill [black] (2,-1) circle (2pt);
\fill [black] (3,-1) circle (2pt);

\end{tikzpicture}

\end{center} 

\caption{\label{fig:pathshortening}An example of a sequence of moves shortening a path with two positive vertices separated by a (possibly empty) set of zero vertices. Sequences of this type are called \emph{chain shortenings}.}

\end{figure}  


\begin{thm}[L.-McCammond: Straightening Paths] Let $X$ be a $\cat(0)$ simplicial $3$-complex, and let $\alpha$ and $\beta$ be combinatorial paths staring at vertex $u$ and ending at a vertex $w$. If $\beta$ is a (possibly empty) geodesic, then $\alpha$ can be reduced to $\beta$ by applying a sequence of  length preserving or length reducing basic moves. If both paths are geodesic, all moves will be length preserving.\end{thm}

\begin{proof} Let $D$ be the disk spanning the combinatorial loop $\alpha\beta^{-1}$. Induct on  the area of $D$, and the length of $\alpha$. If  the area of $D$ is zero,  and $\ell(\alpha) = \ell(\beta)$, then $\alpha = \beta$ and the result holds trivially. If $D$ is nonsingular, then each distinguished vertex has degree at least two.  By the proceeding lemma, the sum of degrees of the nondistinguished vertices along $\beta$ is at most  one.  Thus if $\beta = [v_0=u, v_1, v_2, ...., v_n = w]$, then $$\sum_{v \in \beta^{(0)}} (4-deg_D(v)) = (4 - deg_D(u)) + \sum_{i=1}^{n-1} (4-deg_D(v_i)) + (4-deg_D(w)) \leq 5.$$ This implies that the sum of nondistinguished vertex degrees along $\alpha$ is at least one. Thus $\alpha$ contains a positive vertex, and the corresponding basic move may be applied to $\alpha$. If $D$ is singular, then either $\alpha$ contains a trivial move, or $D$ may be broken into nonsingular subdisks along which a subpath of alpha contains a basic move by the previous argument.\end{proof}

\begin{cor} Every move can be broken in to a sequence of basic moves. \end{cor}

For the remainder of the paper, the term \emph{move} will be used to refer to a \emph{basic move}. The following proposition will be useful in proving regularity results for path systems. In particular, it implies the path system of combinatorial geodesics is regular.

\begin{prop}\label{prop:combfft} Let $\alpha$ be a non geodesic combinatorial path. Then $\alpha$ is $(1,0)$-fellow travelled by a strictly shorter path. \end{prop}

\begin{proof} Let $u$ and $w$ be the distinguished endpoints of $\alpha$. Choose an equivalent geodesic path $\gamma$ such that the doubly based disk $D$ spanning $\alpha\gamma^{-1}$ has minimal area.  Suppose that $D$ is nonsingular. Then each vertex along the boundary of $D$ has degree at least two. If $\gamma$ has a positive vertex on $D$, then the area of $D$ can be reduced by applying the corresponding move. Thus $\gamma_{+}$ is empty, and $\sum_{v \in \gamma} (4-deg_D(v)) \leq (4-deg_D(u)) + (4-deg_D(w)) \leq 4$. This implies the sum of the degree of the nondistinguished vertices along $\alpha$ is at least two. Thus either $\alpha$ contains a nondistinguished vertex of degree two or two non distinguished vertices of degree one separated by a string of zero vertices (see Figure \ref{fig:pathshortening}). In either case $\alpha$ contains a shortening and is $1$-fellow traveled by a shorter path. If $D$ is singular, either $\alpha$ contains a trivial move, or the proceeding argument can be applied to the nonsingular subdisks of $\alpha$.
\end{proof}

The Path Straightening Theorem gives a method to move between two geodesics with common endpoints through the minimal disk spanning their union. Thus the set of combinatorial paths between a distinguished set of points are related via their spanning disks; one can ``sweep'' from one combinatorial path to another combinatorial path with the same distinguished endpoints by applying a sequence of moves. This will be used to determine a path system. The following lemma will be helpful in this process.

\begin{prop}\label{prop:diskunion} Let $D_1$ and $D_2$ be two minimal singular disks with a common combinatorial geodesic $\gamma$ along their boundary. Then $D = D_1 \cup_{\gamma} D_2$ can be reduced to a minimal disk $D' = D'_1 \cup_{\gamma'} D'_2$ where $\gamma'$ is obtained from $\gamma$ by a sequence of triangle-triangle moves. Furthermore, if $v' \in \gamma'$ is the image of $v \in \gamma$, then $deg_{D'_i}(v) \geq deg_{D_i}(v)$ for $i=1,2$. \end{prop}

\begin{proof} If $D$ is not minimal  there exists a vertex $v \in int D$ of degree less than six. $D_1$ and $D_2$ are both minimal, so $v = v_i \in \gamma = [v_0,...,v_n]$. If $deg_{D}(v) = 3$, then $v$ lies on a single triangle on $D_1$ or $D_2$,  contradicting that $\gamma$ is geodesic. Thus either $deg_D(v) = 4$ or $deg_{D}(v) = 5$.

Suppose $deg_D(v) = 4$. Then $\gamma$ separates $\bar{st}_D(v)$ into two triangles on each side, and $deg_{D_1}(v) = deg_{D_2}(v) = 3$. The vertex $v$ has four neighbors in $D$, $v_{i-1}$ and $v_{i+1}$ the vertices proceeding and following $v$ on $\gamma$, $x_1 \in intD_1$, and $x_2 \in intD_2$. The closed combinatorial path $[v_{i-1},x_1,v_{i+1},x_2] $ defines a square in $K$. The complex $K$ contains no empty squares, and $\gamma$ is geodesic, so $x_1$ and $x_2$ must span an edge. This gives two choices of moves; send $v$ to $x_1$ in $D_1$, or  send $v$ to $x_2$ in $D_2$ (see Figure \ref{fig:deg4}).

Suppose the move sending $v$ to $x_1$ is applied. This sends $\gamma$ to $\gamma' =[v_0,...,v_{i-1},x_1,v_{i+1},...,v_n] $, $D$ to $D' = D \setminus (\bar{st}_D(v)) \cup \{v_{i-1}, x_1,x_2\} \cup\{v_{i+1},x_1,x_2\}$,  $D_1$ to $D'_1 = D_1 \setminus (\bar{st}_{D_1}(v))$, and $D_2$ to $D'_2 = D_2 \setminus (\bar{st}_{D_2}(v)) \cup \{v_{i-1}, x_1,x_2\}\cup \{v_{i+1},x_1,x_2\}$. Thus $deg_{D'_1}(x_1) = deg_{D_1} (x_1) - 1 \geq 5 > deg_{D_1}(v)$, and  $deg_{D'_2}(x_1) = deg_{D_2}(v)$. The case that $v$ is sent to $x_2$ is symmetric.


\begin{figure}

\begin{center}

\begin{tikzpicture}[scale=1.5]

\coordinate (A1) at (0.5,0);
\coordinate (B1) at (0, 0.5*3^{0.5} );
\coordinate (C1) at (-0.5,0);
\coordinate (D1) at (0, -0.5*3^{0.5});
\coordinate (O1) at (0,0.1);

\filldraw [blue!20] (A1)--(B1)--(C1)--(D1)--cycle;
\draw (A1)--(B1)--(C1)--(D1)--cycle;
\draw (O1)--(A1);
\draw (O1)--(B1);
\draw (O1)--(C1);
\draw (O1)--(D1);
\draw [dashed] (C1)--(A1);

\filldraw (A1) circle (1pt) node [anchor=west] {$x_2$};
\filldraw (B1) circle (1pt) node [anchor=south] {$v_{i+1}$};
\filldraw (C1) circle (1pt) node [anchor=east] {$x_1$};
\filldraw (D1) circle (1pt) node [anchor=north] {$v_{i-1}$};
\filldraw (O1) circle (1pt) node [anchor=south west] {$v$};

\coordinate (A2) at (2.5,0);
\coordinate (B2) at (2, 0.5*3^{0.5} );
\coordinate (C2) at (1.5,0);
\coordinate (D2) at (2, -0.5*3^{0.5});
\coordinate (O2) at (2,0.1);

\filldraw [blue!20] (A2)--(B2)--(C2)--(D2)--cycle;
\draw (A2)--(B2)--(C2)--(D2)--cycle;
\draw (O2)--(A2);
\draw (O2)--(B2);
\draw (O2)--(C2);
\draw (O2)--(D2);
\draw [dashed] (C2)--(A2);

\filldraw (A2) circle (1pt) node [anchor=west] {$x_2$};
\filldraw (B2) circle (1pt) node [anchor=south] {$v_{i+1}$};
\filldraw (C2) circle (1pt) node [anchor=east] {$x_1$};
\filldraw (D2) circle (1pt) node [anchor=north] {$v_{i-1}$};
\filldraw (O2) circle (1pt) node [anchor=south west] {$v$};

\coordinate (A3) at (0.5,-2.25);
\coordinate (B3) at (0, -2.25+0.5*3^{0.5} );
\coordinate (C3) at (-0.5,-2.25);
\coordinate (D3) at (0, -2.25-0.5*3^{0.5});

\filldraw [blue!20] (A3)--(B3)--(C3)--(D3)--cycle;
\draw (A3)--(B3)--(C3)--(D3)--cycle;
\draw (A3)--(C3);

\filldraw (A3) circle (1pt) node [anchor=west] {$x_2$};
\filldraw (B3) circle (1pt) node [anchor=south] {$v_{i+1}$};
\filldraw (C3) circle (1pt) node [anchor=east] {$v \rightarrow x_1$};
\filldraw (D3) circle (1pt) node [anchor=north] {$v_{i-1}$};

\coordinate (A4) at (2.5,-2.25);
\coordinate (B4) at (2, -2.25+0.5*3^{0.5} );
\coordinate (C4) at (1.5,-2.25);
\coordinate (D4) at (2, -2.25-0.5*3^{0.5});

\filldraw [blue!20] (A4)--(B4)--(C4)--(D4)--cycle;
\draw (A4)--(B4)--(C4)--(D4)--cycle;
\draw (A4)--(C4);

\filldraw (A4) circle (1pt) node [anchor=west] {$v \rightarrow x_2$};
\filldraw (B4) circle (1pt) node [anchor=south] {$v_{i+1}$};
\filldraw (C4) circle (1pt) node [anchor=east] {$ x_1$};
\filldraw (D4) circle (1pt) node [anchor=north] {$v_{i-1}$};

\end{tikzpicture}

\end{center}

\caption{\label{fig:deg4}Two choices for resolving a degree four vertex when taking the union on two minimal disks.}

\end{figure}
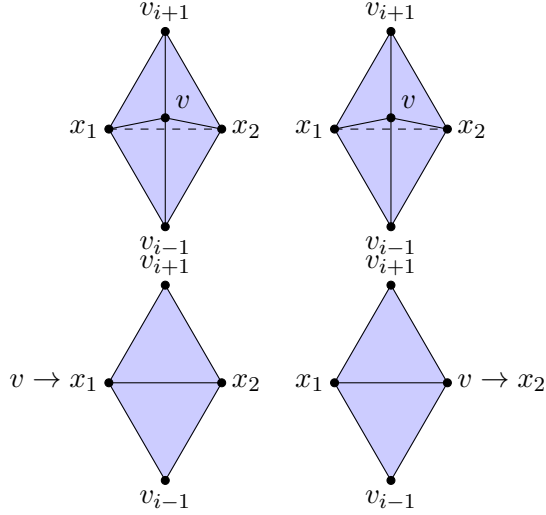

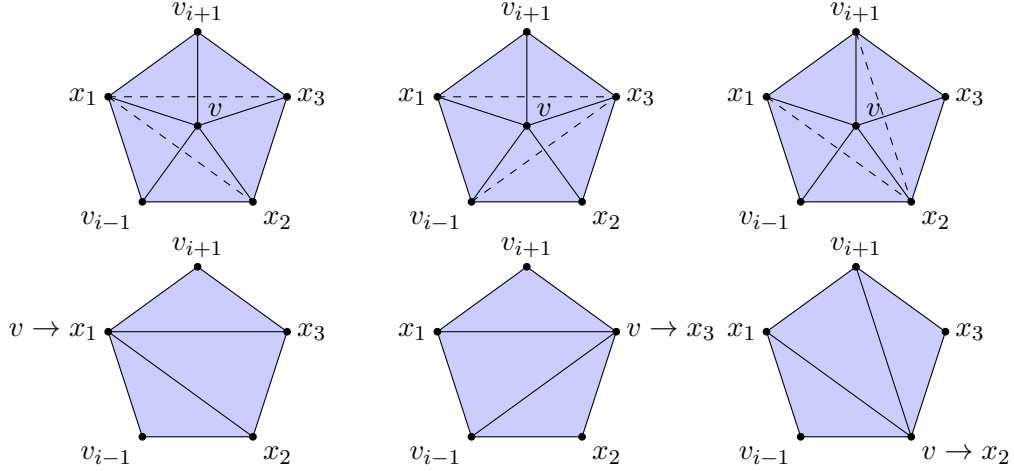
\begin{figure}

\begin{center}

\begin{tikzpicture}[scale=1.25]

\coordinate (A0) at (0,0);
\coordinate (A1) at (0, 1);
\coordinate (A2) at (-0.951056516, 0.309016994);
\coordinate (A3) at (-0.587785252, -0.809016994);
\coordinate (A4) at (0.587785252, -0.809016994);
\coordinate (A5) at (0.951056516, 0.309016994);

\filldraw[blue!20] (A1)--(A2)--(A3)--(A4)--(A5)--cycle;
\draw (A1)--(A2)--(A3)--(A4)--(A5)--cycle;
\draw (A0)--(A1);
\draw (A0)--(A2);
\draw (A0)--(A3);
\draw (A0)--(A4);
\draw (A0)--(A5);

\draw [dashed] (A2)--(A5);
\draw [dashed] (A2)--(A4);

\filldraw (A0) circle (1pt) node [anchor=south west] {$v$};
\filldraw (A1) circle (1pt) node [anchor=south] {$v_{i+1}$};
\filldraw (A2) circle (1pt) node [anchor=east] {$x_1$};
\filldraw (A3) circle (1pt) node [anchor=north east] {$v_{i-1}$};
\filldraw (A4) circle (1pt) node [anchor=north west] {$x_2$};
\filldraw (A5) circle (1pt) node [anchor=west] {$x_3$};

\coordinate (D0) at (0,-2.5);
\coordinate (D1) at (0, -1.5);
\coordinate (D2) at (-0.951056516, -2.5+0.309016994);
\coordinate (D3) at (-0.587785252, -2.5-0.809016994);
\coordinate (D4) at (0.587785252, -2.5-0.809016994);
\coordinate (D5) at (0.951056516, -2.5+0.309016994);

\filldraw[blue!20] (D1)--(D2)--(D3)--(D4)--(D5)--cycle;
\draw (D1)--(D2)--(D3)--(D4)--(D5)--cycle;

\draw  (D2)--(D5);
\draw (D2)--(D4);

\filldraw (D1) circle (1pt) node [anchor=south] {$v_{i+1}$};
\filldraw (D2) circle (1pt) node [anchor=east] {$v \rightarrow x_1$};
\filldraw (D3) circle (1pt) node [anchor=north east] {$v_{i-1}$};
\filldraw (D4) circle (1pt) node [anchor=north west] {$x_2$};
\filldraw (D5) circle (1pt) node [anchor=west] {$x_3$};

\coordinate (B0) at (3.5,0);
\coordinate (B1) at (3.5, 1);
\coordinate (B2) at (3.5-0.951056516, 0.309016994);
\coordinate (B3) at (3.5-0.587785252, -0.809016994);
\coordinate (B4) at (3.5+0.587785252, -0.809016994);
\coordinate (B5) at (3.5+0.951056516, 0.309016994);

\filldraw[blue!20] (B1)--(B2)--(B3)--(B4)--(B5)--cycle;
\draw (B1)--(B2)--(B3)--(B4)--(B5)--cycle;
\draw (B0)--(B1);
\draw (B0)--(B2);
\draw (B0)--(B3);
\draw (B0)--(B4);
\draw (B0)--(B5);

\draw [dashed] (B2)--(B5);
\draw [dashed] (B3)--(B5);

\filldraw (B0) circle (1pt) node [anchor=south west] {$v$};
\filldraw (B1) circle (1pt) node [anchor=south] {$v_{i+1}$};
\filldraw (B2) circle (1pt) node [anchor=east] {$x_1$};
\filldraw (B3) circle (1pt) node [anchor=north east] {$v_{i-1}$};
\filldraw (B4) circle (1pt) node [anchor=north west] {$x_2$};
\filldraw (B5) circle (1pt) node [anchor=west] {$x_3$};

\coordinate (E0) at (3.5,-2.5);
\coordinate (E1) at (3.5, -1.5);
\coordinate (E2) at (3.5-0.951056516, -2.5+0.309016994);
\coordinate (E3) at (3.5-0.587785252, -2.5-0.809016994);
\coordinate (E4) at (3.5+0.587785252, -2.5-0.809016994);
\coordinate (E5) at (3.5+0.951056516, -2.5+0.309016994);

\filldraw[blue!20] (E1)--(E2)--(E3)--(E4)--(E5)--cycle;
\draw (E1)--(E2)--(E3)--(E4)--(E5)--cycle;

\draw  (E2)--(E5);
\draw (E3)--(E5);

\filldraw (E1) circle (1pt) node [anchor=south] {$v_{i+1}$};
\filldraw (E2) circle (1pt) node [anchor=east] {$x_1$};
\filldraw (E3) circle (1pt) node [anchor=north east] {$v_{i-1}$};
\filldraw (E4) circle (1pt) node [anchor=north west] {$x_2$};
\filldraw (E5) circle (1pt) node [anchor=west] {$v \rightarrow x_3$};

\coordinate (C0) at (7,0);
\coordinate (C1) at (7, 1);
\coordinate (C2) at (7-0.951056516, 0.309016994);
\coordinate (C3) at (7-0.587785252, -0.809016994);
\coordinate (C4) at (7+0.587785252, -0.809016994);
\coordinate (C5) at (7+0.951056516, 0.309016994);

\filldraw[blue!20] (C1)--(C2)--(C3)--(C4)--(C5)--cycle;
\draw (C1)--(C2)--(C3)--(C4)--(C5)--cycle;
\draw (C0)--(C1);
\draw (C0)--(C2);
\draw (C0)--(C3);
\draw (C0)--(C4);
\draw (C0)--(C5);

\draw [dashed] (C2)--(C4);
\draw [dashed] (C1)--(C4);

\filldraw (C0) circle (1pt) node [anchor=south west] {$v$};
\filldraw (C1) circle (1pt) node [anchor=south] {$v_{i+1}$};
\filldraw (C2) circle (1pt) node [anchor=east] {$x_1$};
\filldraw (C3) circle (1pt) node [anchor=north east] {$v_{i-1}$};
\filldraw (C4) circle (1pt) node [anchor=north west] {$x_2$};
\filldraw (C5) circle (1pt) node [anchor=west] {$x_3$};

\coordinate (F0) at (7,-2.5);
\coordinate (F1) at (7, -1.5);
\coordinate (F2) at (7-0.951056516, -2.5+0.309016994);
\coordinate (F3) at (7-0.587785252, -2.5-0.809016994);
\coordinate (F4) at (7+0.587785252, -2.5-0.809016994);
\coordinate (F5) at (7+0.951056516, -2.5+0.309016994);

\filldraw[blue!20] (F1)--(F2)--(F3)--(F4)--(F5)--cycle;
\draw (F1)--(F2)--(F3)--(F4)--(F5)--cycle;

\draw  (F2)--(F4);
\draw (F1)--(F4);

\filldraw (F1) circle (1pt) node [anchor=south] {$v_{i+1}$};
\filldraw (F2) circle (1pt) node [anchor=east] {$x_1$};
\filldraw (F3) circle (1pt) node [anchor=north east] {$v_{i-1}$};
\filldraw (F4) circle (1pt) node [anchor=north west] {$v \rightarrow x_2$};
\filldraw (F5) circle (1pt) node [anchor=west] {$ x_3$};

\end{tikzpicture}

\end{center}

\caption{\label{fig:5deg} The three ways to resolve a degree five vertex in the union of two minimal disks. Unlike the degree four case, the choice of move depends on the filling of the pentagon.}

\end{figure}


Suppose $deg_D(v)=5$. Then $v$ has five neighbors in $D$. As before, let $v_{i-1}$ be the neighbor proceeding $v$ in $\gamma$ and $v_{i+1}$ the neighbor following $v$ in $\gamma$. Then $\gamma$ separates $\bar{st}_D(v)$ into two sides, one with two triangles, the other with three. Without loss of generality suppose $\bar{st}_{D_1}(v)$ consists of two triangles $\{x_1,v_{i-1}, v\}$ and $\{x_1,v, v_{i+1}\}$, and  $\bar{st}_{D_2}(v)$ contains three triangles $\{v_{i-1},x_2,v\}$, $\{x_2,x_3,v\}$ and $\{x_3,v_{i+1},v\}$. Thus $deg_{D_1}(v) = 3$ and $deg_{D_2}(v) = 4$. The path $[x_1,v_{i-1},x_2,x_3, v_{i+1}]$ forms a pentagon in $K$, so two pairs of opposing vertices span an edge in $K$. Since $\gamma$ is geodesic, $v_{i-1}$, $v_{i+1}$ can not span an edge. This leaves three possible pairs of edges: $\{x_1,x_2\}$ and $\{x_1,x_3\}$, $\{x_1,x_3\}$ and $\{v_{i-1},x_3\}$, and $\{x_1,x_2\}$ and $\{v_{i+1},x_2\}$(see Figure \ref{fig:5deg}). 

If the pair $\{x_1,x_2\}$,$\{x_1,x_3\}$ span edges, apply the move at $v$ in $D_1$ sending $v$ to $v'=x_1$. The disk $D_1$ is sent to $D'_1 = D_1 \setminus \bar{st}{D_1}(v)$ and $D_2$ is sent to $D_2' = D_2 \setminus \bar{st}_{D_2}(v) \cup \{x_1,v_{i-1},x_2\}\cup \{x_1, x_2, x_3\}\cup\{x_1,x_3, v_{i+1}\}$. Thus $deg_{D'_1}(x_1) = deg_{D'_1}(x_1) - 1 \geq 5 > deg_{D_1}(v)$ and $deg_{D'_2}(x_1) = 4 = deg_{D_2} (v)$. If the pair $\{x_1,x_3\}$ and $\{v_{i-1},x_3\}$ is filled apply the move sending $v$ to $v' = x_3$,  and  if the filled pair is $\{x_1,x_2\}$ and $\{v_{i+1},x_2\}$ apply the move $v$ to $v' = x_2$. In each case the degree of the image of $v$ in $D_1$ is either preserved or increased. \end{proof}


\section{Gersten-Short Geodesics}\label{sec:gsgeos}


In this section the canonical path system used to prove Main Theorem \ref{main:bi} is defined. The paths are a generalization of those used by Gersten and Short in their proof of Theorem \ref{thm:gs1}.

\begin{defn}[GS-Geodesic] A combinatorial path $\gamma = [v_0,v_1,...v_n]$ is a \emph{Gersten-Short geodesic} or \emph{GS-geodesic} if for each minimal singular disk $D$ with $\gamma$ along the boundary of $D$ and $i=1,...n-1$ the following conditions hold.

\begin{enumerate}

\item Geodesic Condition: $\gamma$ is a combinatorial geodesic.
\item Diagonal Condition: If $v_{i} \in \gamma_+$ then $v_{i-1} \in \gamma_{-}$.

\end{enumerate}

Condition 2 forces the paths to move diagonally whenever possible. Pairs of vertices on combinatorial geodesics not satisfying the diagonal condition are called \emph{bad pairs}. Denote the path system of GS geodesic by $\mathcal{P}_{GS}$.\end{defn}

\begin{lem} \label{lem:diag1} Let $\gamma=[v_0,v_1,...,v_n]$ be a combinatorial geodesic lying along the boundary of a minimal singluar disk $D$. Then there exists a finite sequence of triangle-triangle moves sending $\gamma$ and $D$ to a combinatorial geodesic $\gamma'$ and disk $D'$ satisfying the diagonal condition. \end{lem}

\begin{proof} The degree of each vertex of a combinatorial geodesic along any disk is at least three. If $v_{i-1}, v_{i} \in \gamma_{+}(D)$ then they define a shortening. Thus only pairs of the form $v_i \in \gamma_{+}(D)$, $v_{i-1} \in \gamma_{0}(D)$ need to be avoided. 

Reading from $v_0$ to $v_n$ along $D$, look for a pair $v_{i-1}$, $v_i$ such that $v_i \in \gamma_{+}(D)$, $v_{i-1} \in \gamma_{0}(D)$. Then $v_i$ lies on a triangle-triangle move in $D$. Apply this move to obtain a new combinatorial geodesic $\gamma' = [v_0,v_1,...,v_{i-1},v'_{i},v_{i+1},...,v_n]$ and $D$ to $D' = D \setminus st_D(v_i)$. Then $deg_{D'}(v'_i) \geq 4$, so $v'_i \in \gamma'_{-}$. Replace $\gamma$ with $\gamma'$ and $D$ with $D'$. Now $v_{i-1} \in \gamma'_{+}$ (see Figure \ref{fig:resolve}). This move might cause a chain of triangle-triangle moves to be applied at each $v_j$ such that  $j < i$ and $v_j \in \gamma_0(D)$. At most resolving each bad pair will require $\sum_{i=1}^n i = \frac{n(n-1)}{2}$ steps. \end{proof}


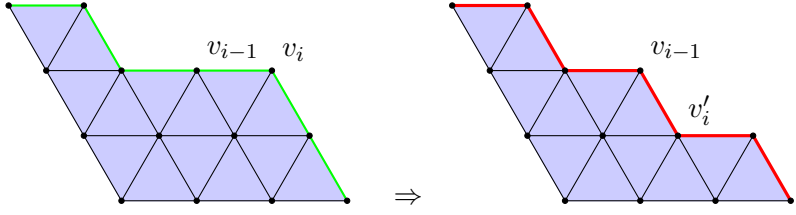
\begin{figure}

\begin{tabular}{ccc}

\begin{tikzpicture}

\coordinate (A1) at (1,0);
\coordinate (A2) at (2,0);
\coordinate (A3) at (3,0);
\coordinate (A4) at (4,0);

\coordinate (B1) at (1-0.5,0.5*3^{0.5});
\coordinate (B2) at (2-0.5,0.5*3^{0.5});
\coordinate (B3) at (3-0.5,0.5*3^{0.5});
\coordinate (B4) at (4-0.5,0.5*3^{0.5});

\coordinate (C1) at (1-1,3^{0.5});
\coordinate (C2) at (2-1,3^{0.5});
\coordinate (C3) at (3-1,3^{0.5});
\coordinate (C4) at (4-1,3^{0.5});

\coordinate (D1) at (1-1.5,1.5*3^{0.5});
\coordinate (D2) at (2-1.5,1.5*3^{0.5});

\filldraw[blue!20] (A1)--(A4)--(C4)--(C2)--(D2)--(D1)--cycle;

\draw(D1)--(A1)--(A4);
\draw(B1)--(B4);
\draw(C1)--(C2);
\draw(C2)--(A2);
\draw(C3)--(A3);
\draw (C1)--(D2);
\draw (B1)--(C2);
\draw(A1)--(C3);
\draw(A2)--(C4);
\draw(A3)--(B4);

\draw[green, thick] (D1)--(D2)--(C2)--(C4)--(A4);

\filldraw (A1) circle (1pt) node [anchor=south west] {$ $};
\filldraw (A2) circle (1pt) node [anchor=south west] {$ $};
\filldraw (A3) circle (1pt) node [anchor=south west] {$ $};
\filldraw (A4) circle (1pt) node [anchor=south west] {$ $};

\filldraw (B1) circle (1pt) node [anchor=south west] {$ $};
\filldraw (B2) circle (1pt) node [anchor=south west] {$ $};
\filldraw (B3) circle (1pt) node [anchor=south west] {$ $};
\filldraw (B4) circle (1pt) node [anchor=south west] {$ $};

\filldraw (C1) circle (1pt) node [anchor=south west] {$ $};
\filldraw (C2) circle (1pt) node [anchor=south west] {$ $};
\filldraw (C3) circle (1pt) node [anchor=south west] {$v_{i-1}$};
\filldraw (C4) circle (1pt) node [anchor=south west] {$v_i$};

\filldraw (D1) circle (1pt) node [anchor=south west] {$ $};
\filldraw (D2) circle (1pt) node [anchor=south west] {$ $};

\end{tikzpicture} & $\Rightarrow$ &

\begin{tikzpicture}

\coordinate (A1) at (1,0);
\coordinate (A2) at (2,0);
\coordinate (A3) at (3,0);
\coordinate (A4) at (4,0);

\coordinate (B1) at (1-0.5,0.5*3^{0.5});
\coordinate (B2) at (2-0.5,0.5*3^{0.5});
\coordinate (B3) at (3-0.5,0.5*3^{0.5});
\coordinate (B4) at (4-0.5,0.5*3^{0.5});

\coordinate (C1) at (1-1,3^{0.5});
\coordinate (C2) at (2-1,3^{0.5});
\coordinate (C3) at (3-1,3^{0.5});
\coordinate (C4) at (4-1,3^{0.5});

\coordinate (D1) at (1-1.5,1.5*3^{0.5});
\coordinate (D2) at (2-1.5,1.5*3^{0.5});

\filldraw[blue!20] (A1)--(A4)--(B4)--(B3)--(C3)--(C2)--(D2)--(D1)--cycle;

\draw(D1)--(A1)--(A4);
\draw(B1)--(B3);
\draw(C1)--(C2);
\draw(C2)--(A2);
\draw(C3)--(A3);
\draw (C1)--(D2);
\draw (B1)--(C2);
\draw(A1)--(C3);
\draw(A2)--(B3);
\draw(A3)--(B4);

\draw[red, very thick] (D1)--(D2)--(C2)--(C3)--(B3)--(B4)--(A4);

\filldraw (A1) circle (1pt) node [anchor=south west] {$ $};
\filldraw (A2) circle (1pt) node [anchor=south west] {$ $};
\filldraw (A3) circle (1pt) node [anchor=south west] {$ $};
\filldraw (A4) circle (1pt) node [anchor=south west] {$ $};

\filldraw (B1) circle (1pt) node [anchor=south west] {$ $};
\filldraw (B2) circle (1pt) node [anchor=south west] {$ $};
\filldraw (B3) circle (1pt) node [anchor=south west] {$v'_i $};
\filldraw (B4) circle (1pt) node [anchor=south west] {$ $};

\filldraw (C1) circle (1pt) node [anchor=south west] {$ $};
\filldraw (C2) circle (1pt) node [anchor=south west] {$ $};
\filldraw (C3) circle (1pt) node [anchor=south west] {$v_{i-1}$};

\filldraw (D1) circle (1pt) node [anchor=south west] {$ $};
\filldraw (D2) circle (1pt) node [anchor=south west] {$ $};

\end{tikzpicture} \\

\end{tabular}

\caption{\label{fig:resolve} Resolving a bad pair in the proof of Lemma \ref{lem:diag1}.}

\end{figure}


\begin{lem} \label{lem:diag2} Let $\gamma$ be a combinatorial geodesic lying along the boundaries of two minimal, singular disks $D_1$ and $D_2$ such that $\gamma$ satisfies the diagonal condition on $D_1$. Suppose a triangle-triangle move is applied to $\gamma$ on $D_2$ to resolve a bad pair. Then the altered path $\gamma'$ satisfies the diagonal condition on $D'_1$.   \end{lem}

\begin{proof}   By Proposition \ref{prop:diskunion}, we may assume that $D_1 \cup D_2$ is minimal. Suppose a triangle-triangle move has been applied to $\gamma$ sent the vertex $v_j$ to the vertex $v_j'$, $\gamma$ to $\gamma' = [v_0,..,v_{j-1},v'_j,v_{j+1},..,v_n]$, $D_2$ to $D'_2 = D_2 \setminus st_{D_2}(v_j)$, and $D_1$ to $D'_1 = D_1 \cup \{[v_{j-1},v_j, v'_j],[v_{j+1},v_j,v'_j]\}$. It suffices to check pairs involving the vertices $v_{j-1}, v'_j, v_{j+1}$ are not bad pairs on $(\gamma',D_1')$.  

If $(v_{j-1},v_{j})$ was a bad pair on $D_2$, then $v_j \in \gamma_+(D_2)$ and $v_{j-1} \in \gamma_0(D_2)$. After the move, $v_j' \in \gamma'_-(D'_2)$ and $v_{j-1} \in \gamma_+(D'_2)$. By minimality, $v_{j-1}$ is contained in at least four triangles in $D'_1$, so $v_{j-1} \in \gamma_-(D'_1)$ and is not a member of a bad pair. Since $v'_j \in \gamma'_+(D'_1)$ and $\gamma'$ is geodesic, $v_{j+1} \in \gamma'_0(D'_1) \cup \gamma'_-(D_1')$, so $(v'_j, v_{j+1})$ is not a bad pair. Finally, suppose $v_{j+1} \in \gamma'_0(D'_1)$ and $v_{j+2} \in \gamma'_+(D'_1)$. The triangle-triangle move preserved the degree of $v_{j+2}$ and reduced the degree of $v_{j+1}$ by one. This implies that $v_{j+1}, v_{j+2} \in \gamma_+(D_1)$, a contradiction to $\gamma$ geodesic. Thus $(v_{j+1}, v_{j+2})$ is also not a bad pair of $\gamma'$ along $D'_1$.

\end{proof}


\begin{figure}

\begin{tabular}{cc}

\begin{tikzpicture}[scale=1.125]

\coordinate (A1) at (1,0);
\coordinate (A2) at (2,0);
\coordinate (A3) at (3,0);

\coordinate (B1) at (1-0.5,0.5*3^{0.5});
\coordinate (B2) at (2-0.5,0.5*3^{0.5});
\coordinate (B3) at (3-0.5,0.5*3^{0.5});

\coordinate (C1) at (1-1,3^{0.5});
\coordinate (C2) at (2-1,3^{0.5});
\coordinate (C3) at (3-1,3^{0.5});

\filldraw[blue!10] (-1,-0.5)--(4,-0.5)--(4,3^{0.5}+0.75)--(-1, 3^{0.5}+0.75)--cycle;

\filldraw[blue!20] (-1,-0.5)--(4,-0.5)--(A3)--(C3)--(C1)--(-1,3^{0.5}+0.75)--cycle;

\draw[green, thick] (4,-0.5)--(A3)--(C3)--(C1)--(-1,3^{0.5}+0.75);

\draw (A1)--(A3);
\draw (B1)--(B3);
\draw (B1)--(C2);
\draw (A1)--(C3);
\draw(A2)--(B3);
\draw(A1)--(C1);
\draw(A2)--(C2);

\filldraw (A1) circle (1pt) node [anchor=south west] {$ $};
\filldraw (A2) circle (1pt) node [anchor=south west] {$ $};
\filldraw (A3) circle (1pt) node [anchor=south west] {$v_{j+2}$};

\filldraw (B1) circle (1pt) node [anchor=south west] {$ $};
\filldraw (B2) circle (1pt) node [anchor=south west] {$ $};
\filldraw (B3) circle (1pt) node [anchor=south west] {$v_{j+1}$};

\filldraw (C1) circle (1pt) node [anchor=south west] {$v_{j-2}$};
\filldraw (C2) circle (1pt) node [anchor=south west] {$v_{j-1}$};
\filldraw (C3) circle (1pt) node [anchor=south west] {$v_{j}$};

\draw (3.9, 3^{0.5}+0.65) node [anchor = north east] {$D_1$};
\draw (-0.9, -0.4) node [anchor = south west] {$D_2$};

\end{tikzpicture} &

\begin{tikzpicture}[scale=1.125]

\coordinate (A1) at (1,0);
\coordinate (A2) at (2,0);
\coordinate (A3) at (3,0);

\coordinate (B1) at (1-0.5,0.5*3^{0.5});
\coordinate (B2) at (2-0.5,0.5*3^{0.5});
\coordinate (B3) at (3-0.5,0.5*3^{0.5});

\coordinate (C1) at (1-1,3^{0.5});
\coordinate (C2) at (2-1,3^{0.5});
\coordinate (C3) at (3-1,3^{0.5});

\filldraw[blue!10] (-1,-0.5)--(4,-0.5)--(4,3^{0.5}+0.75)--(-1, 3^{0.5}+0.75)--cycle;

\filldraw[blue!20] (-1,-0.5)--(4,-0.5)--(A3)--(B3)--(B2)--(C2)--(C1)--(-1,3^{0.5}+0.75)--cycle;

\draw[red, very thick] (4,-0.5)--(A3)--(B3)--(B2)--(C2)--(C1)--(-1,3^{0.5}+0.75);

\draw (A1)--(A3);
\draw (B1)--(B2);
\draw (B1)--(C2);
\draw (A1)--(C3);
\draw(A2)--(B3);
\draw(A1)--(C1);
\draw(A2)--(B2);
\draw(B3)--(C3)--(C2);

\filldraw (A1) circle (1pt) node [anchor=south west] {$ $};
\filldraw (A2) circle (1pt) node [anchor=south west] {$ $};
\filldraw (A3) circle (1pt) node [anchor=south west] {$v_{j+2}$};

\filldraw (B1) circle (1pt) node [anchor=south west] {$ $};
\filldraw (B2) circle (1pt) node [anchor= north] {$v'_j $};
\filldraw (B3) circle (1pt) node [anchor=south west] {$v_{j+1}$};

\filldraw (C1) circle (1pt) node [anchor=south west] {$v_{j-2}$};
\filldraw (C2) circle (1pt) node [anchor=south west] {$v_{j-1}$};
\filldraw (C3) circle (1pt) node [anchor=south west] {$ $};

\draw (3.9, 3^{0.5}+0.65) node [anchor = north east] {$D'_1$};
\draw (-0.9, -0.4) node [anchor = south west] {$D'_2$};

\end{tikzpicture} \\

\end{tabular}

\caption{\label{fig:diag2} Disk configuration described in the proof of Lemma \ref{lem:diag2}.}

\end{figure}
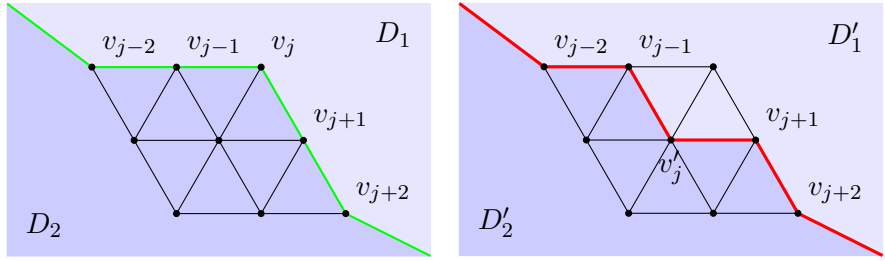

\begin{thm}\label{thm:trans} Every pair of vertices in a $\cat(0)$ simplicial $3$-complex $K$ is connected by a GS-geodesic. \end{thm}

\begin{proof}Start with any combinatorial geodesic $\gamma$ between $u$ and $w$. Locally $\gamma$ bounds in finitely many spanning disks $D$. If $\gamma$ does not satisfy the diagonal condition on one of these disks, apply Lemma \ref{lem:diag1} to obtain a combinatorial geodesic $\gamma'$ from $u$ to $w$ that satisfies the diagonal condition on $D'$. Replace $\gamma$ with $\gamma'$ and repeat. By Lemma \ref{lem:diag2}, this process does not create bad pairs. There are finitely many combinatorial geodesics from $u$ to $w$, so this process will eventually halt. \end{proof}

\begin{thm}[Fellow Travelling] $\mathcal{P}_{GS}$ satisfies the $(2,1)$-fellow traveller property. \end{thm}

\begin{figure}

\begin{center}

\begin{tikzpicture}[scale=1.25]

\coordinate (O1) at (0,0);
\coordinate (A1) at (0.5*3^{0.5},0.5);
\coordinate (B1) at (0,1);
\coordinate (C1) at (-0.5*3^{0.5},0.5);
\coordinate (D1) at (-0.5*3^{0.5},-0.5);
\coordinate (E1) at (0,-1);
\coordinate (F1) at (0.5*3^{0.5},-0.5);

\coordinate (O2) at (6,0);
\coordinate (A2) at (6+0.5*3^{0.5},0.5);
\coordinate (B2) at (6,1);
\coordinate (C2) at (6-0.5*3^{0.5},0.5);
\coordinate (D2) at (6-0.5*3^{0.5},-0.5);
\coordinate (E2) at (6,-1);
\coordinate (F2) at (6+0.5*3^{0.5},-0.5);

\draw (-0.5*3^{0.5},0) node [anchor = east] {$a$};
\draw (6+0.5*3^{0.5},0) node [anchor=west] {$b$};

\filldraw[blue!20] (A1)--(B1)--(C1)--(D1)--(E1)--(F1)--(D2)--(E2)--(F2)--(A2)--(B2)--(C2)--cycle;

\draw (O1)--(B1)--(C1)--(D1)--(E1)--cycle;
\draw (O2)--(B2)--(A2)--(F2)--(E2)--cycle;

\draw[dashed] (O1)--(A1);
\draw (O1)--(C1);
\draw (O1)--(D1);
\draw[dashed] (O1)--(F1);

\draw[dashed] (O2)--(C2);
\draw (O2)--(A2);
\draw (O2)--(F2);
\draw[dashed] (O2)--(D2);

\draw[dashed] (B1)--(A1)--(C2)--(B2);
\draw[dashed] (E1)--(F1)--(D2)--(E2);

\filldraw (O1) circle (1pt) node [anchor=south west] {$ $};
\filldraw (A1) circle (1pt) node [anchor=south west] {$ $};
\filldraw (B1) circle (1pt) node [anchor=south] {$\delta(1)$};
\filldraw (C1) circle (1pt) node [anchor=south east] {$\delta(0)$};
\filldraw (D1) circle (1pt) node [anchor=north east] {$\gamma(0)$};
\filldraw (E1) circle (1pt) node [anchor=north] {$\gamma(1)$};
\filldraw (F1) circle (1pt) node [anchor=north west] {$ $};

\filldraw (O2) circle (1pt) node [anchor=south west] {$ $};
\filldraw (A2) circle (1pt) node [anchor=south west] {$\delta(m)$};
\filldraw (B2) circle (1pt) node [anchor=south] {$\delta(m-1)$};
\filldraw (C2) circle (1pt) node [anchor=south east] {$ $};
\filldraw (D2) circle (1pt) node [anchor=north east] {$ $};
\filldraw (E2) circle (1pt) node [anchor=north] {$\gamma(n-1)$};
\filldraw (F2) circle (1pt) node [anchor=north west] {$\gamma(n)$};

\end{tikzpicture}

\end{center}

\caption{The configuration when $\alpha$ contains six chains. \label{fig:hexcase}}

\end{figure}
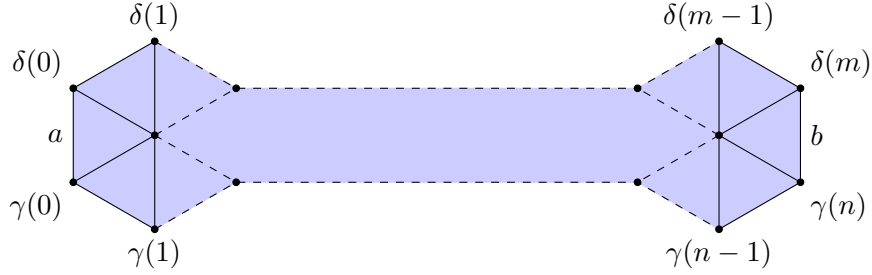

\begin{proof} Let $\delta$ and $\gamma$ be paths in  $\mathcal{P}_{GS}$ on with endpoints at most distance one apart. Let $\alpha$ be the closed combinatorial path consisting of $\gamma$ followed by $\delta^{-1}$, together with the possible edges separating their endpoints. Let $D$ be the minimal disk spanning $\alpha$. The proof follows by induction on the length of $\alpha$.

If $\ell(\alpha) \leq 5$, then the path distance between $\delta$ and $\gamma$ is at most two by the no empty triangle, square, and pentagon condition. Suppose the induction hypothesis holds for pairs of GS-geodesics defining boundary paths of length less than $n$, and $\ell(\alpha) = n$. If $deg_{D}(\gamma(0))\leq 2$, then $\gamma(0)$ lies on an edge or a single triangle in $D$ and $d(\gamma(1),\delta(0))\leq 1$. Then the boundary path defined by $\gamma' = [\gamma(1),...,\gamma(m)]$ and $\delta$ satisfy the induction hypothesis. Thus the path distance $d(\gamma,\delta) = d(\gamma',\delta) \leq 2$. The same argument applies to each endpoint of $\gamma$ and $\delta$.

Suppose each endpoints of $\gamma$ and $\delta$ has degree at least three on $D$. Then $\alpha$ contains no vertices of degree two. By Lemma \ref{lem:+vecount}, $\alpha$ contains at least six move positive vertices than negative vertices. Thus $\alpha$ contains at least six chain shortenings. Both $\gamma$ and $\delta$ are $GS$-geodesics, so each chain intersects either the edge $a$ or the edge $b$, and  has length one by the diagonal condition (see Figure \ref{fig:hexcase}). Thus $d_C(\gamma(1), \delta(1)) = 2$ and  $deg_D(\delta(0)) = deg_D(\gamma(0)) = deg_D(\gamma(1)) = 3$. Let $\beta = [\delta(0), \gamma(0), \gamma(1),...,\gamma(m)]$. Then $\beta$ is a non geodesic path containing the move $[\delta(0), \gamma(0), \gamma(1), \gamma(2)]$. Take $\beta' $ to be the image of $\beta$ after applying the move. By the previous case, $d(\beta'(t),\delta(t)) \leq 2$ for all $t$.  For $t >1$, $\gamma(t) = \beta'(t)$. Thus $d(\gamma(t),\delta(t))\leq 2$ for all $t$.
\end{proof}


\section{Biautomaticity of Proper Actions}\label{sec:biautomaticity}


This section is devoted to the proof of Main Theorem \ref{main:bi}. Let $G$ be  a group acting geometrically by isometries on a $\cat(0)$ simplicial $3$-complex $K$. 

Choose a distinguished vertex $v_0$ of $K$ and restrict $\mathcal{P}_{GS}$ to paths between vertices of the orbit $G_{v_0}$ to produce the path system $\mathcal{P}_{GS,v_0}$. The action of $G$ on $K$ induces an action on the $1$-skeleton of $K$. The action preserves the length of combinatorial paths, and the degree of vertices. Thus it sends $\mathcal{P}_{GS,v_0}$ to $\mathcal{P}_{GS,v_0}$, and $\mathcal{P}_{GS,v_0}$ is $G$-invariant. The fellow traveller property still holds when we restrict ourselves to paths in the orbit of $v_0$. All that remains to prove our main theorem is to show is path system is regular.

\begin{prop} $\mathcal{P}_{GS,v_0}$ is regular. \end{prop}

\begin{proof} The falsification by fellow traveller property still holds when restricted to paths between vertices of $G_{v_0}$, so the path system of $\mathcal{G}_{v_0}$ of combinatorial geodesics between vertices of $G_{v_0}$ is regular by Proposition \ref{prop:combfft} and Theorem \ref{thm:fft}. The action of $G$ on $K$ is geometric, so the quotient space $K/G$ is compact and thus finite.  There finitely many paths $\beta = [v_0,v_1,v_2,v_3]$ such that $v_2 \in \beta_{+}$ and $v_1 \in \beta _0$. Let $S$ be that set of such subpaths $\beta$. For each $\beta \in S$, build an FSA $M_{\beta}$ that accepts the path $\epsilon$ if and only if $\epsilon$ contains a lift of $\beta$ as a subpath. Then the path system $\mathcal{S}_{\beta}$ of paths containing a lift of $\beta \in S$ is accepted by $M_{\beta}$ and thus regular. This implies that $\mathcal{P}_{GS,v_0} = \mathcal{G}_{v_0} \setminus (\bigcup_{\beta \in S} M_{\beta})$ is regular. \end{proof}

\renewcommand{\themain}{\ref{main:bi}}
\begin{main}Groups acting properly on $\cat(0)$ simplicial $3$-complexes are biautomatic.
\end{main}
\renewcommand{\themain}{\Alph{main}}


\section{Higher Dimensions}\label{sec:highdim}


Unfortunately this technique can not be directly generalized to higher dimensional simplicial complexes. The main stumbling block is trying to generalize Crowley's minimal spanning disk theorem to higher dimensions. While simplicial $n$-complexes for $n\geq4$ satisfy the no empty triangle and no empty square condition, they might have empty pentagons.


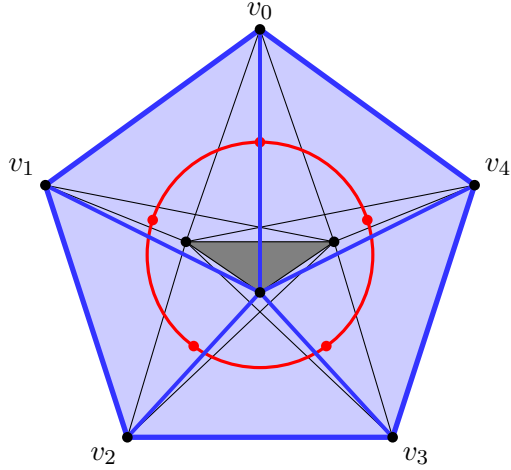
\begin{figure}

\begin{tikzpicture}

\coordinate (O) at (0,0);
\coordinate (P0) at (90: 3);
\coordinate (P1) at (90+72: 3);
\coordinate (P2) at (90+2*72: 3);
\coordinate (P3) at (90+3*72: 3);
\coordinate (P4) at (90+4*72: 3);

\coordinate (T0) at (270 + 100: 1);
\coordinate (T1) at (270 - 100: 1);
\coordinate (T2) at (270: 0.5);

\coordinate (L0) at (90: 1.5);
\coordinate (L1) at (90+72: 1.5);
\coordinate (L2) at (90+2*72: 1.5);
\coordinate (L3) at (90+3*72: 1.5);
\coordinate (L4) at (90+4*72: 1.5);

\filldraw[blue!20] (P0)--(P1)--(P2)--(P3)--(P4)--cycle;

\draw[blue!80, line width=2 pt] (P0)--(P1)--(P2)--(P3)--(P4)--cycle;

\filldraw[black!50] (T0)--(T1)--(T2)--cycle;

\draw (T0)--(T1)--(T2)--cycle;

\draw (T0)--(P0);
\draw (T0)--(P1);
\draw (T0)--(P2);
\draw (T0)--(P3);
\draw (T0)--(P4);

\draw (T1)--(P0);
\draw (T1)--(P1);
\draw (T1)--(P2);
\draw (T1)--(P3);
\draw (T1)--(P4);

\draw[red, very thick] (O) circle (1.5);

\fill[red] (L0) circle (2pt);
\fill[red] (L1) circle (2pt);
\fill[red] (L2) circle (2pt);
\fill[red] (L3) circle (2pt);
\fill[red] (L4) circle (2pt);

\draw[blue!80, line width=1.5 pt] (T2)--(P0);
\draw[blue!80, line width=1.5 pt] (T2)--(P1);
\draw[blue!80, line width=1.5 pt] (T2)--(P2);
\draw[blue!80, line width=1.5 pt] (T2)--(P3);
\draw[blue!80, line width=1.5 pt] (T2)--(P4);

\fill[black] (P0) circle (2pt) node [anchor=south] {$v_0$};
\fill[black] (P1) circle (2pt) node [anchor=south east] {$v_1$};
\fill[black] (P2) circle (2pt) node [anchor=north east] {$v_2$};
\fill[black] (P3) circle (2pt) node [anchor=north west] {$v_3$};
\fill[black] (P4) circle (2pt) node [anchor=south west] {$v_4$};

\fill[black] (T0) circle (2pt);
\fill[black] (T1) circle (2pt);
\fill[black] (T2) circle (2pt);

\end{tikzpicture}

\caption{The complex described in Example \ref{exmp:baddisk} for $n=4$. \label{fig:baddisk}}

\end{figure}


\begin{exmp}[High Dimensional Empty Pentagon]\label{exmp:baddisk} Let $\sigma$ be an $(n-2)$-dimensional simplex for $n \geq 4$, and take $K$ to be the join of  $\sigma$ together with closed cycle of five distinct vertices $\alpha = [v_0,v_2,...,v_5=v_0]$. Then $K$ is the union of five distinct $n$-simplices arranged cyclically around $\sigma$ (see Figure \ref{fig:baddisk}).  The curvature of $K$ depends only on $lk_K(\sigma)$. The link is a regular metric graph consisting of a single cycle $c$ with five edges. Each edge has length $\arccos(\frac{1}{n})$, giving $\ell(c) = 5\arccos(\frac{1}{n}) > 2\pi$ for $n \geq 4$.  $K$ is also simply connected, so $K$ is $\cat(0)$.  There are $n-1$ minimal disks spanning $\alpha$ in $K$, one through each of the vertices of $\sigma$. Each disk is an example of a full subcomplex that is positively curved; the disks consist of five triangles around a interior vertex. \end{exmp} 

\noindent Excluding configurations like those described in Example \ref{exmp:baddisk} give complexes which fall under Januszkiewicz and \'{S}wi\c{a}tkowski's theory of simplicial nonpositive curvature \cite{JaSw06}. The loss of $\cat(0)$ spanning disks eliminates many of the tools used to prove the Main Theorem, such as Corollary \ref{cor:bdcount}. 


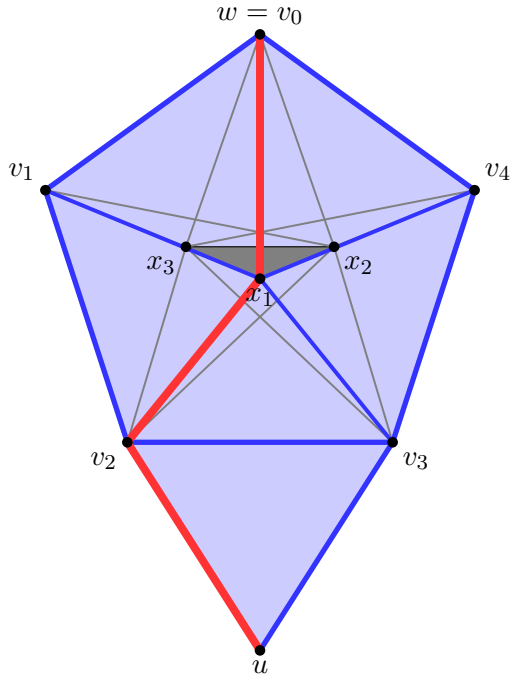
\begin{figure}

\begin{tikzpicture}

\coordinate (O) at (0,0);
\coordinate (P0) at (90: 3);
\coordinate (P1) at (90+72: 3);
\coordinate (P2) at (90+2*72: 3);
\coordinate (P3) at (90+3*72: 3);
\coordinate (P4) at (90+4*72: 3);
\coordinate (P5) at (0, -3*1.73205081);

\coordinate (T0) at (270 + 100: 1);
\coordinate (T1) at (270 - 100: 1);
\coordinate (T2) at (270: 0.25);

\coordinate (L0) at (90: 1.5);
\coordinate (L1) at (90+72: 1.5);
\coordinate (L2) at (90+2*72: 1.5);
\coordinate (L3) at (90+3*72: 1.5);
\coordinate (L4) at (90+4*72: 1.5);

\filldraw[blue!20] (P0)--(P1)--(P2)--(P5)--(P3)--(P4)--cycle;

\draw[blue!80, line width=2 pt] (P0)--(P1)--(P2)--(P3)--(P4)--cycle;

\draw[blue!80, line width=2 pt] (P2)--(P5)--(P3);

\filldraw[black!50] (T0)--(T1)--(T2)--cycle;

\draw (T0)--(T1)--(T2)--cycle;

\draw[black!50, line width=0.75 pt] (T0)--(P0);
\draw[black!50, line width=0.75 pt] (T0)--(P1);
\draw[black!50, line width=0.75 pt] (T0)--(P2);
\draw[black!50, line width=0.75 pt] (T0)--(P3);
\draw[black!50, line width=0.75 pt] (T0)--(P4);

\draw[black!50, line width=0.75 pt] (T1)--(P0);
\draw[black!50, line width=0.75 pt] (T1)--(P1);
\draw[black!50, line width=0.75 pt] (T1)--(P2);
\draw[black!50, line width=0.75 pt] (T1)--(P3);
\draw[black!50, line width=0.75 pt] (T1)--(P4);

\draw[blue!80, line width=1.5 pt] (T2)--(P0);
\draw[blue!80, line width=1.5 pt] (T2)--(P1);
\draw[blue!80, line width=1.5 pt] (T2)--(P2);
\draw[blue!80, line width=1.5 pt] (T2)--(P3);
\draw[blue!80, line width=1.5 pt] (T2)--(P4);

\draw[red!80, line width=3 pt] (P5)--(P2)--(T2)--(P0);

\fill[black] (P0) circle (2pt) node [anchor=south] {$w = v_0$};
\fill[black] (P1) circle (2pt) node [anchor = south east] {$v_1$};
\fill[black] (P2) circle (2pt) node [anchor = north east] {$v_2$};
\fill[black] (P3) circle (2pt)node [anchor = north west] {$v_3$};
\fill[black] (P4) circle (2pt)node [anchor = south west] {$v_4$};
\fill[black] (P5) circle (2pt)node [anchor=north] {$u$};

\fill[black] (T0) circle (2pt) node [anchor = north west] {$x_2$};
\fill[black] (T1) circle (2pt)node [anchor = north east] {$x_3$};
\fill[black] (T2) circle (2pt) node [anchor = north] {$x_1$};

\end{tikzpicture}

\caption{The vertices $x_1$ and $y_2$ form a bad pair on the combinatorial geodesic $\gamma = [u, x_1, y_2, v]$. Take $D$ to be a disk containing the triangles spanned by $\{u, v_2, v_3\}$, $\{v_2, v_3, x_3\}$, $\{v_2, x_1, x_3\}$, and $\{x_3, x_1, v\}$. Then $y_2 \in \gamma_+(D)$, but $x_1 \in \gamma_0(D)$.\label{fig:narwahl1}}

\end{figure}


 Without $\cat(0)$ spanning disks it is unclear how to generalize GS-geodesics to higher dimensions. Consider the $\cat(0)$ $n$-complex $K$ obtained by attaching a triangle along the lower edge of the high dimensional empty pentagon in Example \ref{exmp:baddisk} (see Figure \ref{fig:narwahl1}). There are $2n$ combinatorial geodesics of length four from $u$  to $v$ in $K$. The middle two vertices of each path form a bad pair.

\bibliography{researchbib}

\begin{thebibliography}{1}

\bibitem{Bo92}
B.~H. Bowditch.
\newblock Notes on locally {${\rm CAT}(1)$} spaces.
\newblock In {\em Geometric group theory ({C}olumbus, {OH}, 1992)}, volume~3 of
  {\em Ohio State Univ. Math. Res. Inst. Publ.}, pages 1--48. de Gruyter,
  Berlin, 1995.

\bibitem{BrHa99}
Martin~R. Bridson and Andre Haefliger.
\newblock {\em Metric spaces of non-positive curvature}.
\newblock Springer-Verlag Berlin Heidelberg, New York, NY, 1999.

\bibitem{Cr08}
Katherine Crowley.
\newblock Simplicial collapsibility, discrete {M}orse theory, and the geometry
  of nonpositively curved simplicial complexes.
\newblock {\em Geom. Dedicata}, 133:35--50, 2008.

\bibitem{ECHLPT92}
David B.~A. Epstein, James~W. Cannon, Derek~F. Holt, Silvio V.~F. Levy,
  Michael~S. Paterson, and William~P. Thurston.
\newblock {\em Word processing in groups}.
\newblock Jones and Bartlett Publishers, Boston, MA, 1992.

\bibitem{GeSh91}
S.~M. Gersten and H.~B. Short.
\newblock Small cancellation theory and automatic groups.
\newblock {\em Invent. Math.}, 102(2):305--334, 1990.

\bibitem{JaSw06}
Tadeusz Januszkiewicz and Jacek {\'S}wi{\c{a}}tkowski.
\newblock Simplicial nonpositive curvature.
\newblock {\em Publ. Math. Inst. Hautes \'Etudes Sci.}, (104):1--85, 2006.

\bibitem{LeMc09}
R.~{Levitt} and J.~{McCammond}.
\newblock {Triangles, squares and geodesics}.
\newblock {\em ArXiv e-prints}, October 2009.

\bibitem{Sw06}
Jacek {\'S}wi{\.a}tkowski.
\newblock Regular path systems and ({B}i)automatic groups.
\newblock {\em Geom. Dedicata}, 118:23--48, 2006.

\end{thebibliography}

\bibliographystyle{plain}

\end{document}